\documentclass[letterpaper]{article}%
\usepackage{amsmath}
\usepackage{amsfonts}
\usepackage{amssymb}
\usepackage{graphicx}
\usepackage[colorlinks=true, pdfstartview=FitV, linkcolor=blue, citecolor=red]%
{hyperref}
\usepackage{color}
\usepackage[compress]{cite}%
\providecommand{\U}[1]{\protect\rule{.1in}{.1in}}
\newtheorem{theorem}{Theorem}

\newtheorem{lemma}[theorem]{Lemma}

\newtheorem{remark}[theorem]{Remark}

\newenvironment{proof}[1][Proof]{\noindent\textbf{#1.} }{\ \rule{0.5em}{0.5em}}
\numberwithin{equation}{section}
\begin{document}

\title{Generalizations of the Euler-Mascheroni constant associated with the
hyperharmonic numbers}
  \author{M\"{u}m\"{u}n Can, Ayhan Dil, Levent Karg\i n,\\
  Mehmet Cenkci, Mutlu G\"{u}lo\u{g}lu
 \\ \small{Department of Mathematics, Akdeniz University, TR-07058 Antalya, Turkey}\\
\small{mcan@akdeniz.edu.tr,  adil@akdeniz.edu.tr,  lkargin@akdeniz.edu.tr,}\\
\small{cenkci@akdeniz.edu.tr,  guloglu@akdeniz.edu.tr}}
\maketitle
\begin{abstract}
In this paper, we present two new generalizations of the Euler-Mascheroni
constant arising from the Dirichlet series associated to the hyperharmonic
numbers. We also give some inequalities related to upper and lower estimates,
and evaluation formulas.


\end{abstract}

\section{Introduction}

The Euler-Mascheroni constant $\gamma=0.577\,215\,664\,9\ldots$ occurs in
estimating the growth rate of the harmonic series:%
\begin{equation}
H_{n}=1+\frac{1}{2}+\cdots+\frac{1}{n}\approx\ln n+\gamma.\label{H1}%
\end{equation}
The importance of the constant $\gamma$\ goes beyond its definition as it
turns up in analysis, number theory, probability, and special functions. For
instance, we have $\psi\left(  1\right)  =\Gamma^{^{\prime}}\left(  1\right)
=-\gamma$, where%
\[
\Gamma\left(  x\right)  =\int_{0}^{\infty}t^{x-1}e^{-t}dt,\text{ }x>0
\]
is Euler's gamma function and $\psi\left(  x\right)  =\Gamma^{^{\prime}%
}\left(  x\right)  /\Gamma\left(  x\right)  $ is the digamma function.
$\gamma$ is also related to the Riemann zeta function%
\begin{equation}
\zeta\left(  s\right)  =\sum_{n=1}^{\infty}\frac{1}{n^{s}},\text{
}\operatorname{Re}s>1,\label{Z}%
\end{equation}
in that (see \cite{SC} and \cite{GKP})%
\[
\gamma=\sum_{k=2}^{\infty}\frac{\left(  -1\right)  ^{k}\zeta\left(  k\right)
}{k}\text{ and }\gamma=1-\sum_{k=2}^{\infty}\frac{\zeta\left(  k\right)
-1}{k}.
\]
It might be worthful to record that the Euler-Mascheroni constant is nothing
but the sum of the Mascheroni series (see \cite{Ma})%
\[
\gamma=\sum_{k=2}^{\infty}\frac{\left\vert b_{k}\right\vert }{k},
\]
where $b_{k}$ are the Bernoulli numbers of the second kind (or the Gregory
coefficients) \cite{Bl, J}.

The definition and representations above provide several important
generalizations for the Euler-Mascheroni constant. For instance,
interpretating (\ref{H1})\textbf{ }as%
\[
\gamma=\lim_{n\rightarrow\infty}\left(  \sum_{k=1}^{n}\frac{1}{k}-\int_{1}%
^{n}\frac{1}{x}dx\right)
\]
leads to a motivation to consider the limit\textbf{ }%
\begin{equation}
\lim_{n\rightarrow\infty}\left(  \sum_{k=1}^{n}f\left(  k\right)  -\int%
_{1}^{n}f\left(  x\right)  dx\right)  \label{MaC}%
\end{equation}
(note that if $f:\left(  0,\infty\right)  \rightarrow\left(  0,\infty\right)
$ is continuous, strictly decreasing and $\lim\limits_{x\rightarrow\infty
}f\left(  x\right)  =0,$ then the limit (\ref{MaC}) exists (cf. \cite{We})).
Certain functions are then found to be of special importance in connection
with the study of $\gamma$. Evidently, the function $f\left(  x\right)  =1/x$
in (\ref{MaC}) corresponds to the constant $\gamma$. The function $f\left(
x\right)  =1/x^{s}$ with $0<s\leq1$ gives the so called generalized
Euler-Mascheroni constants $\gamma^{\left(  s\right)  }$:%
\[
\gamma^{\left(  s\right)  }=\lim_{n\rightarrow\infty}\left(  \sum_{k=1}%
^{n}\frac{1}{k^{s}}-\int_{1}^{n}\frac{1}{x^{s}}dx\right)  ,\text{ }0<s\leq1.
\]
The Stieltjes constant $\gamma_{m}$ arises from the choice of $f\left(
x\right)  =\left(  \ln^{m}x\right)  /x$, namely%
\[
\gamma_{m}=\lim_{n\rightarrow\infty}\left(  \sum_{k=1}^{n}\frac{\ln^{m}k}%
{k}-\int_{1}^{n}\frac{\ln^{m}x}{x}dx\right)  .
\]
The Stieltjes constant and the Riemann zeta function share a particular
relationship\textbf{ }in that%
\[
\zeta\left(  s\right)  =\frac{1}{s-1}+\sum_{k=0}^{\infty}\frac{\left(
-1\right)  ^{k}}{k!}\gamma_{k}\left(  s-1\right)  ^{k},
\]
i.e., $\gamma_{m}$ appears in the Laurent expansion of $\zeta\left(  s\right)
$.

The constants $\gamma$, $\gamma^{\left(  s\right)  }$, and $\gamma_{m}$ can in
fact be obtained when we consider the multivariate function $f\left(
x,s\right)  =1/x^{s}$. More precisely, the limit (\ref{MaC}) corresponds to%
\[
\left\{
\begin{array}
[c]{ll}%
\gamma, & \text{for }f=f\left(  x,1\right)  ,\\
\gamma^{(s)}, & \text{for }f=f\left(  x,s\right)  \text{ with }0<s\leq1,\\
\gamma_{m}, & \text{for }f=\left.  \frac{\partial^{m}}{\partial s^{m}}f\left(
x,s\right)  \right\vert _{s=1}.
\end{array}
\right.
\]

It is seen that these constants emerge in the interval $0<s\leq1$ where the
series (\ref{Z}) is divergent. This observation suggests to consider
generalizations of the Riemann zeta function for general versions of the
Euler-Mascheroni constant. The classical example in this direction is the
Hurwitz zeta function%
\[
\zeta\left(  s,a\right)  =\sum_{n=0}^{\infty}\frac{1}{\left(  n+a\right)
^{s}},\text{ }\operatorname{Re}\left(  s\right)  >1,\text{ }a\neq
0,-1,-2,\ldots.
\]
This series diverges when $s\leq1,$ hence $f=f(x,a,s)=1/\left(  x+a\right)
^{s}$ in (\ref{MaC}) with $s=1$ brings out the generalized Euler-Mascheroni
constant $\gamma\left(  a\right)  $, that is,%
\[
\gamma\left(  a\right)  =\lim_{n\rightarrow\infty}\left(  \sum_{k=0}^{n}%
\frac{1}{k+a}-\int_{0}^{n}\frac{1}{x+a}dx\right)
\]
(see \cite[p. 453]{Kn}). Accordingly, $f=\left.  \frac{\partial^{m}}{\partial
s^{m}}\frac{1}{\left(  x+a\right)  ^{s}}\right\vert _{s=1}$ in (\ref{MaC})
introduces the generalized Stieltjes constants $\gamma_{m}(a)$, i.e.,%
\[
\gamma_{m}\left(  a\right)  =\lim_{n\rightarrow\infty}\left(  \sum_{k=0}%
^{n}\frac{\ln^{m}\left(  k+a\right)  }{k+a}-\int_{0}^{n}\frac{\ln^{m}\left(
x+a\right)  }{x+a}dx\right)  .
\]
We note that the constants $\gamma_{m}\left(  a\right)  $ occur in the Laurent
series expansion of the Hurwitz zeta function (see \cite{Be,Wi}).

Generalizations of the Euler-Mascheroni constant have been studied
extensively. They appear in the evaluation of series and integrals containing
some important functions (see, for example, \cite{AS, BlC, Bo, BK,Briggs,
CaC,CK, Co, Co2019, CY, DBA,FPS,Lehmer,MS,SC}). Besides their lower and upper
bounds are studied (see, for example
\cite{Be,Bl-1,DeT,Ka,La,Lu,LuSY,Mo,S2007,S2013,Yo}) and some evaluation
formulas are given (see, for example \cite{Bl,JB,SoH,Xu}).

In this study, we consider the following Dirichlet series (so called the Euler
sums of the hyperharmonic numbers \cite{DB,MD} and also see \cite{KCDC})%
\begin{equation}
\sum_{k=1}^{\infty}\frac{h_{k}^{\left(  r\right)  }}{k^{s}}\text{ and }%
\sum_{k=1}^{\infty}\frac{h_{k}^{\left(  r\right)  }}{k^{\overline{s}}},
\label{eshhn}%
\end{equation}
which are generalizations of $\zeta\left(  s\right)  $. Here $x^{\overline{r}%
}=\Gamma\left(  x+r\right)  /\Gamma\left(  x\right)  $ for $x,r+x\in
\mathbb{C}\backslash\left\{  0,-1,-2,...\right\}  $, and $h_{n}^{(r)}$ are the
hyperharmonic numbers defined by \cite{CG}%
\[
h_{n}^{(r)}=\sum_{k=1}^{n}h_{k}^{(r-1)}\text{ with\ }h_{n}^{(0)}=\frac{1}%
{n},\text{ }n,r\in\mathbb{N}.
\]
It is known that both series in (\ref{eshhn})\ are convergent for $s>r$ and
divergent for $0<s\leq r$ (see \cite{MD}). We mainly focus on the limit
(\ref{MaC}) for $f=f\left(  x,r\right)  =h_{x}^{\left(  r\right)  }/x^{r}$ and
$f=f\left(  x,r\right)  =h_{x}^{\left(  r\right)  }/x^{\overline{r}}$, where
$h_{x}^{\left(  r\right)  }$ is the analytic extension of $h_{n}^{(r)}$,
defined by \cite{M}
\begin{equation}
h_{x}^{\left(  r\right)  }=\frac{x^{\overline{r}}}{x\Gamma\left(  r\right)
}\left(  \psi\left(  x+r\right)  -\psi\left(  r\right)  \right)  ,\text{
}r,x,r+x\in\mathbb{C}\backslash\left\{  0,-1,-2,...\right\}  . \label{2}%
\end{equation}

We first introduce the family of constants corresponding to the function
$f=f\left(  x,r\right)  =h_{x}^{\left(  r\right)  }/x^{r}$ in the limit
(\ref{MaC}). Let $\left\{  y_{n}\left(  r\right)  \right\}  _{n=1}^{\infty}$
and $\left\{  z_{n}\left(  r\right)  \right\}  _{n=1}^{\infty}$ be sequences
defined by
\[
y_{n}\left(  r\right)  =\sum_{k=1}^{n}\frac{h_{k}^{\left(  r\right)  }}{k^{r}%
}-\int_{1}^{n}\frac{h_{x}^{\left(  r\right)  }}{x^{r}}dx\text{ and }%
z_{n}\left(  r\right)  =\sum_{k=1}^{n-1}\frac{h_{k}^{\left(  r\right)  }%
}{k^{r}}-\int_{1}^{n}\frac{h_{x}^{\left(  r\right)  }}{x^{r}}dx.
\]
We prove the following:

\begin{theorem}
\label{mteo1}Let $r\in\left[  0,+\infty\right)  $. Then,

\noindent a) The sequences $\left\{  y_{n}\left(  r\right)  \right\}
_{n=1}^{\infty}$ and $\left\{  z_{n}\left(  r\right)  \right\}  _{n=1}%
^{\infty}$ satisfy
\[
0\leq z_{n}\left(  r\right)  <z_{n+1}\left(  r\right)  <\cdots<y_{n+1}\left(
r\right)  <y_{n}\left(  r\right)  \leq1\text{, for each }n\in\mathbb{N}%
\text{.}%
\]

\noindent b) The sequences $\left\{  y_{n}\left(  r\right)  \right\}
_{n=1}^{\infty}$ and $\left\{  z_{n}\left(  r\right)  \right\}  _{n=1}%
^{\infty}$ converge to the common limit denoted by $\gamma_{h^{\left(
r\right)  }}$, i.e.,
\[
\gamma_{h^{\left(  r\right)  }}=\lim_{n\rightarrow\infty}z_{n}\left(
r\right)  =\lim_{n\rightarrow\infty}y_{n}\left(  r\right)  .
\]

\noindent c) For each $n\in\mathbb{N}$, we have the following estimates:
\[
A\left(  n,r\right)  <\gamma_{h^{\left(  r\right)  }}-z_{n}\left(  r\right)
<A\left(  n,r\right)  +\frac{h_{n+1}^{\left(  r\right)  }}{\left(  n+1\right)
^{r}}%
\]
\ and\
\[
B\left(  n,r\right)  -\frac{h_{n+1}^{\left(  r\right)  }}{\left(  n+1\right)
^{r}}<y_{n}\left(  r\right)  -\gamma_{h^{\left(  r\right)  }}<B\left(
n,r\right)  ,
\]
where $A\left(  n,r\right)  =h_{n}^{\left(  r\right)  }/n^{r}-\int_{n}%
^{n+1}\left(  h_{x}^{\left(  r\right)  }/x^{r}\right)  dx$ and $B\left(
n,r\right)  =\int_{n}^{n+1}\left(  h_{x}^{\left(  r\right)  }/x^{r}\right)
dx.$
\end{theorem}

We next consider the function $f=f\left(  x,r\right)  =h_{x}^{\left(
r\right)  }/x^{\overline{r}}$ in the limit (\ref{MaC}). Let $\left\{
a_{n}\left(  r\right)  \right\}  _{n=1}^{\infty}$ and $\left\{  b_{n}\left(
r\right)  \right\}  _{n=1}^{\infty}$ be sequences defined by
\[
a_{n}\left(  r\right)  =\sum_{k=1}^{n-1}\frac{h_{k}^{\left(  r\right)  }%
}{k^{\overline{r}}}-\int_{1}^{n}\frac{h_{x}^{\left(  r\right)  }}%
{x^{\overline{r}}}dx\text{ \ and \ }b_{n}\left(  r\right)  =\sum_{k=1}%
^{n}\frac{h_{k}^{\left(  r\right)  }}{k^{\overline{r}}}-\int_{1}^{n}%
\frac{h_{x}^{\left(  r\right)  }}{x^{\overline{r}}}dx.
\]
These sequences satisfy the following properties.

\begin{theorem}
\label{mteo2}Let $r\in\left[  0,+\infty\right)  $. Then,

\noindent a) The sequences $\left\{  a_{n}\left(  r\right)  \right\}
_{n=1}^{\infty}$ and $\left\{  b_{n}\left(  r\right)  \right\}  _{n=1}%
^{\infty}$ satisfy
\[
0\leq a_{n}\left(  r\right)  <a_{n+1}\left(  r\right)  <\cdots<b_{n+1}\left(
r\right)  <b_{n}\left(  r\right)  \leq1/\Gamma\left(  r+1\right)  \text{, for
all }n\in\mathbb{N}\text{.}%
\]

\noindent b) The sequences $\left\{  a_{n}\left(  r\right)  \right\}
_{n=1}^{\infty}$ and $\left\{  b_{n}\left(  r\right)  \right\}  _{n=1}%
^{\infty}$ converge to the common limit denoted by $\overline{\gamma
}_{h^{\left(  r\right)  }}$, i.e.,%
\[
\overline{\gamma}_{h^{\left(  r\right)  }}=\lim_{n\rightarrow\infty}%
a_{n}\left(  r\right)  =\lim_{n\rightarrow\infty}b_{n}\left(  r\right)  .
\]

\noindent c) For each $n\in\mathbb{N}$, we have the following estimates:
\[
C\left(  n,r\right)  <\overline{\gamma}_{h^{\left(  r\right)  }}-a_{n}\left(
r\right)  <C\left(  n,r\right)  +\frac{h_{n+1}^{\left(  r\right)  }}{\left(
n+1\right)  ^{\overline{r}}}%
\]
\ and\
\[
D\left(  n,r\right)  -\frac{h_{n+1}^{\left(  r\right)  }}{\left(  n+1\right)
^{\overline{r}}}<b_{n}\left(  r\right)  -\overline{\gamma}_{h^{\left(
r\right)  }}<D\left(  n,r\right)  ,
\]
where $C\left(  n,r\right)  =h_{n}^{\left(  r\right)  }/n^{\overline{r}}%
-\int_{n}^{n+1}\left(  h_{x}^{\left(  r\right)  }/x^{\overline{r}}\right)  dx$
and $D\left(  n,r\right)  =\int_{n}^{n+1}\left(  h_{x}^{\left(  r\right)
}/x^{\overline{r}}\right)  dx$.

\noindent d) For $r\geq1$ we have $\overline{\gamma}_{h^{\left(  r\right)  }%
}>\overline{\gamma}_{h^{\left(  r+1\right)  }}$.
\end{theorem}

In Section 2 we prove these results. In Section 3 we establish formulas to
calculate $\gamma_{h^{\left(  r\right)  }}$ and $\overline{\gamma}_{h^{\left(
r\right)  }}$ (see Theorem \ref{teo1} and Theorem \ref{teo1-}). As a
demonstration, we list the first few values of $\gamma_{h^{\left(  r\right)
}}$ and $\overline{\gamma}_{h^{\left(  r\right)  }}$ here:\label{liste}%
\[
\hspace{-0.09in}%
\begin{tabular}
[c]{ll}%
$\gamma_{h^{\left(  0\right)  }}=\gamma\simeq0.577\,215\,664\,9,$ &
$\overline{\gamma}_{h^{\left(  0\right)  }}=\gamma\simeq0.577\,215\,664\,9,$%
\smallskip\\
$\gamma_{h^{\left(  1\right)  }}\simeq0.529\,052\,969\,9,$ & $\overline
{\gamma}_{h^{\left(  1\right)  }}=\gamma_{h^{\left(  1\right)  }}%
\simeq0.529\,052\,969\,9,$\smallskip\\
$\gamma_{h^{\left(  2\right)  }}\simeq0.555\,196\,054\,9,$ & $\overline
{\gamma}_{h^{\left(  2\right)  }}\simeq0.258\,690\,124\,4,$\smallskip\\
$\gamma_{h^{\left(  3\right)  }}\simeq0.597\,861\,674\,3,$ & $\overline
{\gamma}_{h^{\left(  3\right)  }}\simeq0.085\,388\,073\,83,$\smallskip\\
$\gamma_{h^{\left(  4\right)  }}\simeq0.643\,901\,835\,0,$ & $\overline
{\gamma}_{h^{\left(  4\right)  }}\simeq0.021\,230\,652\,79,$\smallskip\\
$\gamma_{h^{\left(  5\right)  }}\simeq0.688\,191\,320\,8,$ & $\overline
{\gamma}_{h^{\left(  5\right)  }}\simeq0.004\,231\,410\,657,$\smallskip\\
$\gamma_{h^{\left(  6\right)  }}\simeq0.728\,430\,828\,1,$ & $\overline
{\gamma}_{h^{\left(  6\right)  }}\simeq0.000\,703\,545\,796,$\smallskip\\
$\gamma_{h^{\left(  7\right)  }}\simeq0.763\,731\,244\,8,$ & $\overline
{\gamma}_{h^{\left(  7\right)  }}\simeq0.000\,100\,330\,361,$\smallskip\\
$\gamma_{h^{\left(  8\right)  }}\simeq0.793\,996\,044\,8,$ & $\overline
{\gamma}_{h^{\left(  8\right)  }}\simeq1.\,252\,451\,689\times10^{-5}%
,$\smallskip\\
$\gamma_{h^{\left(  9\right)  }}\simeq0.819\,563\,050\,9.$ & $\overline
{\gamma}_{h^{\left(  9\right)  }}\simeq1.\,390\,145\,776\times10^{-6}.$%
\end{tabular}
\ \ \ \ \ \
\]

We finally note that since%
\[
\lim_{r\rightarrow0}h_{x}^{\left(  r\right)  }=\lim_{r\rightarrow0}%
\frac{x^{\overline{r}}}{x\Gamma\left(  r\right)  }\left(  \psi\left(
x+r\right)  -\psi\left(  r\right)  \right)  =-\frac{1}{x}\lim_{r\rightarrow
0}\frac{\psi\left(  r\right)  }{\Gamma\left(  r\right)  }=\frac{1}{x},
\]
we use $h_{x}^{\left(  0\right)  }=1/x$ in the sequel.

\section{Proofs of Theorems}

For the proofs of Theorem \ref{mteo1} and Theorem \ref{mteo2} we need the
following lemma.

\begin{lemma}
\label{soru}Let $r>0$ and $0<x<y$. Then,%
\[
\frac{\psi\left(  y+r\right)  -\psi\left(  r\right)  }{\psi\left(  x+r\right)
-\psi\left(  r\right)  }<\frac{y}{x}.
\]

\end{lemma}

\begin{proof}
We start by setting%
\[
\alpha\left(  x,r\right)  =\frac{\psi\left(  x+r\right)  -\psi\left(
r\right)  }{x}%
\]
for $r>0,$ $x>0$. We show that $\alpha\left(  x,r\right)  $ is a decreasing
function with respect to $x$. Differentiating with respect to $x$ gives%
\[
\frac{d}{dx}\alpha\left(  x,r\right)  =\frac{x\frac{d}{dx}\psi\left(
x+r\right)  -\psi\left(  x+r\right)  +\psi\left(  r\right)  }{x^{2}}.
\]
Let $\beta\left(  x,r\right)  =x\frac{d}{dx}\psi\left(  x+r\right)
-\psi\left(  x+r\right)  +\psi\left(  r\right)  $.\ Then,
\[
\frac{d}{dx}\beta\left(  x,r\right)  =x\psi^{\prime\prime}\left(  x+r\right)
<0
\]
implies that the function $\beta\left(  x,r\right)  $ is decreasing with
respect to $x$. Since
\[
\lim_{x\rightarrow0^{+}}\beta\left(  x,r\right)  =\lim_{x\rightarrow0^{+}%
}\left(  x\frac{d}{dx}\psi\left(  x+r\right)  -\psi\left(  x+r\right)
+\psi\left(  r\right)  \right)  =0
\]
and $\beta\left(  x,r\right)  $ is decreasing, we conclude that
\[
\beta\left(  x,r\right)  =x\frac{d}{dx}\psi\left(  x+r\right)  -\psi\left(
x+r\right)  +\psi\left(  r\right)  <0.
\]
This shows that the function $\alpha\left(  x,r\right)  $ is decreasing with
respect to $x$, which completes the proof.
\end{proof}

\subsection{Proof of Theorem \ref{mteo1}}

a) We first show that the sequence $\left\{  y_{n}\left(  r\right)  \right\}
_{n=1}^{\infty}$ is decreasing. From the definition, we have%
\[
y_{n+1}\left(  r\right)  -y_{n}\left(  r\right)  =\frac{h_{n+1}^{\left(
r\right)  }}{\left(  n+1\right)  ^{r}}-\int_{n}^{n+1}\frac{h_{x}^{\left(
r\right)  }}{x^{r}}dx.
\]
The mean value theorem for definite integrals guarantees that there is an
$N\in\left(  n,n+1\right)  $ such that
\[
\int_{n}^{n+1}\frac{h_{x}^{\left(  r\right)  }}{x^{r}}dx=\frac{h_{N}^{\left(
r\right)  }}{N^{r}}%
\]
holds. Then, using (\ref{2}) it is seen that
\begin{align*}
&  y_{n+1}\left(  r\right)  -y_{n}\left(  r\right)  =\frac{h_{n+1}^{\left(
r\right)  }}{\left(  n+1\right)  ^{r}}-\frac{h_{N}^{\left(  r\right)  }}%
{N^{r}}\\
&  =\frac{N^{r+1}\left(  n+1\right)  ^{\overline{r}}\left(  \psi\left(
n+1+r\right)  -\psi\left(  r\right)  \right)  -\left(  n+1\right)
^{r+1}N^{\overline{r}}\left(  \psi\left(  N+r\right)  -\psi\left(  r\right)
\right)  }{\Gamma\left(  r\right)  \left(  n+1\right)  ^{r+1}N^{r+1}}.
\end{align*}
For $r=0,$ the inequality $y_{n+1}\left(  0\right)  -y_{n}\left(  0\right)
<0$\ is trivial since $h_{x}^{\left(  0\right)  }=1/x.$ Let $r>0.$ Then,
\[
y_{n+1}\left(  r\right)  -y_{n}\left(  r\right)  <0\Leftrightarrow\frac
{\psi\left(  n+1+r\right)  -\psi\left(  r\right)  }{\psi\left(  N+r\right)
-\psi\left(  r\right)  }<\frac{n+1}{N}\frac{\left(  n+1\right)  ^{r}%
N^{\overline{r}}}{N^{r}\left(  n+1\right)  ^{\overline{r}}}.
\]
According to the Lemma \ref{soru} it remains to prove that
\begin{equation}
1<\frac{\left(  n+1\right)  ^{r}N^{\overline{r}}}{N^{r}\left(  n+1\right)
^{\overline{r}}}\text{ for }r>0\text{ and }n<N<n+1. \label{6}%
\end{equation}
For this purpose, we set
\[
g\left(  x,r\right)  =\frac{\Gamma\left(  x+r\right)  }{x^{r}\Gamma\left(
x\right)  }=\frac{x^{\overline{r}}}{x^{r}}\text{ for }r>0,x>0.
\]
Differentiating with respect to $x$ gives
\[
\frac{d}{dx}g\left(  x,r\right)  =\frac{\Gamma\left(  x+r\right)  }%
{\Gamma\left(  x\right)  x^{r}}\left(  \psi\left(  x+r\right)  -\psi\left(
x\right)  -\frac{r}{x}\right)  .
\]
The substitutions $x\rightarrow1,$ $r\rightarrow x$ and $y\rightarrow r$ in
Lemma \ref{soru} yields%
\[
\frac{\psi\left(  r+x\right)  -\psi\left(  x\right)  }{r}<\frac{1}{x}.
\]
Thus, we deduce that $g\left(  x,r\right)  =\Gamma\left(  x+r\right)  /\left(
x^{r}\Gamma\left(  x\right)  \right)  \ $is decreasing with respect to $x$ for
$x,r>0.$ This completes the proof of $y_{n+1}\left(  r\right)  <y_{n}\left(
r\right)  $.

To show that the sequence $\left\{  z_{n}\left(  r\right)  \right\}
_{n=1}^{\infty}$ is increasing we again appeal to the mean value theorem for
definite integrals and deduce that
\begin{align*}
z_{n+1}\left(  r\right)  -z_{n}\left(  r\right)   &  =\frac{h_{n}^{\left(
r\right)  }}{n^{r}}-\frac{h_{N}^{\left(  r\right)  }}{N^{r}},\text{ }%
N\in\left(  n,n+1\right) \\
&  =\frac{N^{r+1}n^{\overline{r}}\left(  \psi\left(  n+r\right)  -\psi\left(
r\right)  \right)  -n^{r+1}N^{\overline{r}}\left(  \psi\left(  N+r\right)
-\psi\left(  r\right)  \right)  }{\Gamma\left(  r\right)  n^{r+1}N^{r+1}}.
\end{align*}
Since $h_{x}^{\left(  0\right)  }=1/x$ we have $z_{n+1}\left(  r\right)
-z_{n}\left(  r\right)  <0$\ for $r=0$. Let $r>0.$ Then,
\[
z_{n+1}\left(  r\right)  -z_{n}\left(  r\right)  >0\Leftrightarrow\frac
{\psi\left(  n+r\right)  -\psi\left(  r\right)  }{\psi\left(  N+r\right)
-\psi\left(  r\right)  }>\frac{n}{N}\frac{n^{r}N^{\overline{r}}}%
{N^{r}n^{\overline{r}}}.
\]
It is seen from Lemma \ref{soru} and (\ref{6}) that the sequence $\left\{
z_{n}\left(  r\right)  \right\}  _{n=1}^{\infty}$ is increasing.

Now, above facts and
\[
y_{n}\left(  r\right)  -z_{n}\left(  r\right)  =\frac{h_{n}^{\left(  r\right)
}}{n^{r}}>0,\text{ for each }n\in\mathbb{N}%
\]
imply that%
\[
z_{1}\left(  r\right)  <\cdots<z_{n}\left(  r\right)  <z_{n+1}\left(
r\right)  <\cdots<y_{n+1}\left(  r\right)  <y_{n}\left(  r\right)
<\cdots<y_{1}\left(  r\right)  .
\]

b) The sequence $\left\{  z_{n}\left(  r\right)  \right\}  _{n=1}^{\infty}$ is
increasing and bounded above, hence the limit $\lim\limits_{n\rightarrow
\infty}z_{n}\left(  r\right)  $ exists. Similarly the limit $\lim
\limits_{n\rightarrow\infty}y_{n}\left(  r\right)  $ exists. We observe that%
\[
y_{n}\left(  r\right)  -z_{n}\left(  r\right)  =\frac{h_{n}^{\left(  r\right)
}}{n^{r}}\sim\frac{1}{n^{r}}\frac{n^{r-1}\ln n}{\Gamma\left(  r\right)
}\rightarrow0,n\rightarrow\infty,
\]
so $\lim\limits_{n\rightarrow\infty}z_{n}\left(  r\right)  =\lim
\limits_{n\rightarrow\infty}y_{n}\left(  r\right)  $. Thus,%
\begin{equation}
0=z_{1}\left(  r\right)  <\cdots<z_{n}\left(  r\right)  <\cdots<\gamma
_{h^{\left(  r\right)  }}<\cdots<y_{n}\left(  r\right)  <\cdots<y_{1}\left(
r\right)  =1. \label{5}%
\end{equation}

c) The inequalities can be easily seen from (\ref{5}).

\subsection{Proof of Theorem \ref{mteo2}}

We sketch the proof which is similar to the proof of Theorem \ref{mteo1}.

a) From the definition of $\left\{  b_{n}\left(  r\right)  \right\}
_{n=1}^{\infty}$ and the mean value theorem for definite integrals, it is seen
that%
\[
b_{n+1}\left(  r\right)  -b_{n}\left(  r\right)  =\frac{h_{n+1}^{\left(
r\right)  }}{\left(  n+1\right)  ^{\overline{r}}}-\frac{h_{N}^{\left(
r\right)  }}{N^{\overline{r}}},\text{ }N\in\left(  n,n+1\right)  .
\]
Thus, the validity of $b_{n+1}\left(  r\right)  <b_{n}\left(  r\right)  $
follows from (\ref{2}) and Lemma \ref{soru} when $r>0,$ and from
$h_{x}^{\left(  0\right)  }=1/x$ when $r=0$. In a similar way it can be seen
that the sequence $\left\{  a_{n}\left(  r\right)  \right\}  _{n=1}^{\infty}$
is increasing, and%
\[
a_{1}\left(  r\right)  <\cdots<a_{n}\left(  r\right)  <a_{n+1}\left(
r\right)  <\cdots<b_{n+1}\left(  r\right)  <b_{n}\left(  r\right)
<\cdots<b_{1}\left(  r\right)  .
\]

b) The sequence $\left\{  a_{n}\left(  r\right)  \right\}  _{n=1}^{\infty}$ is
increasing and bounded above, hence the limit $\lim\limits_{n\rightarrow
\infty}a_{n}\left(  r\right)  $ exists. Similarly the limit $\lim
\limits_{n\rightarrow\infty}b_{n}\left(  r\right)  $ exists. We observe that
\[
b_{n}\left(  r\right)  -a_{n}\left(  r\right)  =\frac{h_{n}^{\left(  r\right)
}}{n^{\overline{r}}}\sim\frac{\ln\left(  n+r\right)  }{n\Gamma\left(
r\right)  }\rightarrow0,n\rightarrow\infty,
\]
so%
\[
\lim_{n\rightarrow\infty}a_{n}\left(  r\right)  =\lim_{n\rightarrow\infty
}b_{n}\left(  r\right)  =\overline{\gamma}_{h^{\left(  r\right)  }}.
\]
Thus,%
\begin{equation}
0=a_{1}\left(  r\right)  <\cdots<a_{n}\left(  r\right)  <\cdots<\overline
{\gamma}_{h^{\left(  r\right)  }}<\cdots<b_{n}\left(  r\right)  <\cdots
<b_{1}\left(  r\right)  =\frac{1}{1^{\overline{r}}}. \label{12}%
\end{equation}

c) The estimates follow from (\ref{12}).

d) By the definition, we have%
\[
a_{n}\left(  r+1\right)  -a_{n}\left(  r\right)  =\int_{1}^{n}f\left(
x,r\right)  dx-\sum_{k=1}^{n-1}f\left(  k,r\right)  ,
\]
where%
\[
f\left(  x,r\right)  =\frac{h_{x}^{\left(  r\right)  }}{x^{\overline{r}}%
}-\frac{h_{x}^{\left(  r+1\right)  }}{x^{\overline{r+1}}}.
\]
We would like to show that $a_{n}\left(  r+1\right)  -a_{n}\left(  r\right)
<0.$ For a decreasing function $f,$ the sum $\sum_{k=1}^{n-1}f\left(
k,r\right)  $ is greater than the integral $\int_{1}^{n}f\left(  x,r\right)
dx$ because $\sum_{k=1}^{n-1}f\left(  k,r\right)  $ corresponds to the upper
Darboux sum of $f$ with respect to the partition $P=\left\{  x_{k}%
=k:k=1,2,\ldots,n-1\right\}  .$ Therefore, we need to show that $f\left(
x,r\right)  $ is a decreasing function with respect to $x$.

Let $y>x>0.$ Using (\ref{2}) we have
\begin{align*}
f\left(  y,r\right)  -f\left(  x,r\right)   &  =\left(  \frac{\psi\left(
y+r\right)  -\psi\left(  r\right)  }{y}-\frac{\psi\left(  x+r\right)
-\psi\left(  r\right)  }{x}\right)  \frac{\left(  1-\frac{1}{r}\right)
}{\Gamma\left(  r\right)  }\\
&  +\frac{1}{r\Gamma\left(  r+1\right)  }\frac{x-y}{\left(  y+r\right)
\left(  x+r\right)  }.
\end{align*}
Thus, Lemma \ref{soru} and the fact that
\[
\frac{x-y}{\left(  y+r\right)  \left(  x+r\right)  }<0
\]
imply $f\left(  y,r\right)  -f\left(  x,r\right)  <0,$ i.e, the function
$f\left(  x,r\right)  $ is decreasing.

Since $\overline{\gamma}_{h^{\left(  r\right)  }}=\lim_{n\rightarrow\infty
}a_{n}\left(  r\right)  $, we immediately conclude that
\[
\lim_{n\rightarrow\infty}a_{n}\left(  r+1\right)  =\overline{\gamma
}_{h^{\left(  r+1\right)  }}\leq\overline{\gamma}_{h^{\left(  r\right)  }%
}=\lim_{n\rightarrow\infty}a_{n}\left(  r\right)  ,
\]
or equivalently%
\[
\overline{\gamma}_{h^{\left(  1\right)  }}\geq\cdots\geq\overline{\gamma
}_{h^{\left(  r\right)  }}\geq\overline{\gamma}_{h^{\left(  r+1\right)  }}%
\geq\cdots\geq0.
\]

\section{On evaluations of $\gamma_{h^{\left(  r\right)  }}$\ and
$\overline{\gamma}_{h^{\left(  r\right)  }}$}

In this section, we give representations for the constants $\gamma_{h^{\left(
r\right)  }}$\ and $\overline{\gamma}_{h^{\left(  r\right)  }}$\ when $r$ is a
non-negative integer. We recall that the Stirling numbers of the first kind $%
\genfrac{[}{]}{0pt}{0}{r}{j}%
$ are defined by%
\begin{equation}
x^{\overline{r}}=%
{\displaystyle\sum_{j=0}^{r}}
\genfrac{[}{]}{0pt}{0}{r}{j}%
x^{j}, \label{3}%
\end{equation}
and the generalized harmonic numbers $H_{n}^{\left(  r\right)  }$ are defined
by
\[
H_{n}^{\left(  r\right)  }\mathbf{=}\sum_{k=1}^{n}\frac{1}{k^{r}}.
\]
Note that throughout this section, an empty sum is assumed to be zero.

\subsection{The constants $\gamma_{h^{\left(  r\right)  }}$}

We first analyze the sum and integral in the definition of $y_{n}\left(
r\right)  $, respectively. The aforementioned sum may be evaluated as follows:

\begin{lemma}
\label{Top}Let $r$ be a non-negative integer. Then,%
\begin{align}
\Gamma\left(  r\right)  \sum_{k=1}^{n}\frac{h_{k}^{\left(  r\right)  }}%
{k^{r}}  &  =\frac{\left(  H_{n}\right)  ^{2}+H_{n}^{\left(  2\right)  }}%
{2}-\left(  \psi\left(  r\right)  +\gamma\right)  H_{n}\nonumber\\
&  +\sum_{j=0}^{r-1}%
\genfrac{[}{]}{0pt}{}{r}{j}%
\left(  \sum_{k=1}^{n}\frac{H_{k+r-1}}{k^{r+1-j}}-\left(  \psi\left(
r\right)  +\gamma\right)  H_{n}^{\left(  r+1-j\right)  }\right) \nonumber\\
&  +\sum_{j=1}^{r-1}\frac{H_{j}}{j}-H_{r-1}\left(  H_{n+r-1}-H_{n}\right)
+\sum_{j=1}^{r-1}\frac{H_{j-1}}{j+n}. \label{top}%
\end{align}

\end{lemma}

\begin{proof}
Using (\ref{2}), (\ref{3}) and $\psi\left(  r\right)  +\gamma=H_{r-1}$ we have%
\[
h_{k}^{\left(  r\right)  }=\frac{1}{\Gamma\left(  r\right)  }\sum_{j=0}^{r}%
\genfrac{[}{]}{0pt}{}{r}{j}%
k^{j-1}\left(  H_{k+r-1}-\left(  \psi\left(  r\right)  +\gamma\right)
\right)  ,
\]
from which we obtain that%
\begin{align}
\Gamma\left(  r\right)  \sum_{k=1}^{n}\frac{h_{k}^{\left(  r\right)  }}%
{k^{r}}  &  =\sum_{j=0}^{r-1}%
\genfrac{[}{]}{0pt}{}{r}{j}%
\sum_{k=1}^{n}\frac{H_{k+r-1}}{k^{r+1-j}}+\sum_{k=1}^{n}\frac{H_{k+r-1}}%
{k}\nonumber\\
&  -\left(  \psi\left(  r\right)  +\gamma\right)  \sum_{j=0}^{r-1}%
\genfrac{[}{]}{0pt}{}{r}{j}%
H_{n}^{\left(  r+1-j\right)  }-\left(  \psi\left(  r\right)  +\gamma\right)
H_{n}. \label{4}%
\end{align}
It is easy to see that
\begin{equation}
\sum_{k=1}^{n}\frac{H_{k+r-1}}{k}=\sum_{k=1}^{n}\frac{H_{k}}{k}+\sum
_{j=1}^{r-1}\frac{H_{j}}{j}+H_{r-1}H_{n}-\sum_{j=1}^{r-1}\frac{H_{n+j}}%
{j}\nonumber
\end{equation}
and
\[
\sum_{j=1}^{r-1}\frac{H_{n+j}}{j}=H_{r-1}H_{n}-H_{r-1}H_{n}-H_{r-1}\left(
H_{n+r-1}-H_{n}\right)  +\sum_{j=1}^{r-1}\frac{H_{j-1}}{j+n}.
\]
Using these and the identity\ \cite[Lemma 1]{Ad}%
\[
\sum_{k=1}^{n}\frac{H_{k}}{k}=\frac{\left(  H_{n}\right)  ^{2}+H_{n}^{\left(
2\right)  }}{2},
\]
we deduce that%
\begin{equation}
\sum_{k=1}^{n}\frac{H_{k+r-1}}{k}=\frac{\left(  H_{n}\right)  ^{2}%
+H_{n}^{\left(  2\right)  }}{2}+\sum_{j=1}^{r-1}\left(  \frac{H_{j}}{j}%
+\frac{H_{j-1}}{j+n}\right)  -H_{r-1}\left(  H_{n+r-1}-H_{n}\right)  .
\label{top1}%
\end{equation}
Hence, (\ref{top}) follows from (\ref{4}) and (\ref{top1}).
\end{proof}

Now we evaluate the integral in the definition of $y_{n}\left(  r\right)  $.

\begin{lemma}
\label{Int}Let $r$ be a non-negative integer. Then,%
\begin{align}
\Gamma\left(  r\right)  \int_{1}^{n}\frac{h_{x}^{\left(  r\right)  }}{x^{r}%
}dx  &  =\frac{1}{2}\left(  \ln n\right)  ^{2}-\psi\left(  r\right)  \ln
n+\frac{\left(  n+1\right)  ^{\overline{r-1}}}{n^{r}}-r!+\left(  r+\frac{1}%
{2}\right)  \left(  1-\frac{1}{n}\right) \nonumber\\
&  +%
{\displaystyle\sum_{j=0}^{r-1}}
\genfrac{[}{]}{0pt}{0}{r}{j}%
\left\{  \frac{r}{r+1-j}\left(  1-\frac{1}{n^{r+1-j}}\right)  -\frac
{\psi\left(  r\right)  }{r-j}\left(  1-\frac{1}{n^{r-j}}\right)  \right\}
\nonumber\\
&  -\int_{0}^{\infty}\frac{\ln\left(  1+t^{2}\right)  -\ln\left(  1+\left(
t/n\right)  ^{2}\right)  }{\left(  e^{2\pi t}-1\right)  t}dt+%
{\displaystyle\sum_{j=0}^{r-1}}
\genfrac{[}{]}{0pt}{0}{r}{j}%
\int_{1}^{n}\frac{\psi\left(  x+1\right)  }{x^{r+1-j}}dx. \label{int}%
\end{align}

\end{lemma}

\begin{proof}
Using (\ref{2}) and the well-known identity $\psi\left(  x+1\right)
=\psi\left(  x\right)  +1/x$ we find that
\begin{align}
\Gamma\left(  r\right)  \int_{1}^{n}\frac{h_{x}^{\left(  r\right)  }}{x^{r}%
}dx  &  =\int_{1}^{n}\frac{x^{\overline{r}}}{x^{r+1}}\psi\left(  x+1\right)
dx\nonumber\\
&  +\int_{1}^{n}\left(  \frac{1}{x^{r}}\frac{d}{dx}\left(  x+1\right)
^{\overline{r-1}}-\frac{x^{\overline{r}}}{x^{r+1}}\psi\left(  r\right)
\right)  dx.\nonumber
\end{align}
We now consider each of the integrals on the right hand side one by one.

In view of (\ref{3})\ we have\
\[
\int_{1}^{n}\frac{x^{\overline{r}}}{x^{r+1}}\psi\left(  x+1\right)
dx=\int_{1}^{n}\frac{\psi\left(  x+1\right)  }{x}dx+%
{\displaystyle\sum_{j=0}^{r-1}}
\genfrac{[}{]}{0pt}{0}{r}{j}%
\int_{1}^{n}\frac{\psi\left(  x+1\right)  }{x^{r+1-j}}dx.
\]
Thanks to the expression \cite[Eq. 6.3.21]{AS}
\[
\psi\left(  z\right)  =\ln z-\frac{1}{2z}-2\int_{0}^{\infty}\frac{t}{\left(
e^{2\pi t}-1\right)  \left(  t^{2}+z^{2}\right)  }dt,\text{ }\left\vert \arg
z\right\vert <\frac{\pi}{2},
\]
we find
\[
\int_{1}^{n}\frac{\psi\left(  x+1\right)  }{x}dx=\frac{1}{2}\left(  \ln
n\right)  ^{2}-\frac{1}{2n}+\frac{1}{2}-2\int_{0}^{\infty}\int_{1}^{n}\frac
{t}{\left(  e^{2\pi t}-1\right)  \left(  t^{2}+x^{2}\right)  x}dxdt.
\]
Hence we deduce that%
\[
\int_{1}^{n}\frac{\psi\left(  x+1\right)  }{x}dx=\frac{1}{2}\left(  \ln
n\right)  ^{2}-\frac{1}{2n}+\frac{1}{2}-\int_{0}^{\infty}\frac{\ln\left(
1+t^{2}\right)  -\ln\left(  1+\left(  t/n\right)  ^{2}\right)  }{\left(
e^{2\pi t}-1\right)  t}dt.
\]
It is clear that
\[
\int_{1}^{n}\frac{1}{x^{r}}\frac{d}{dx}\left(  x+1\right)  ^{\overline{r-1}%
}dx=\frac{1}{n^{r}}\left(  n+1\right)  ^{\overline{r-1}}-r!+r%
{\displaystyle\sum_{j=0}^{r}}
\genfrac{[}{]}{0pt}{0}{r}{j}%
\frac{1-\frac{1}{n^{r+1-j}}}{\left(  r+1-j\right)  }%
\]
and%
\[
\psi\left(  r\right)  \int_{1}^{n}\frac{x^{\overline{r}}}{x^{r+1}}%
dx=\psi\left(  r\right)  \ln n+\psi\left(  r\right)
{\displaystyle\sum_{j=0}^{r-1}}
\genfrac{[}{]}{0pt}{0}{r}{j}%
\frac{1}{r-j}\left(  1-\frac{1}{n^{r-j}}\right)  .
\]
The results above yield (\ref{int}).
\end{proof}

We now consider the limit case of the integral%
\[
\int_{1}^{n}\frac{\psi\left(  x+1\right)  }{x^{p}}dx
\]
in (\ref{int}), in which we require the constant $\sigma_{p}$:%
\[
\sigma_{p}=\sum_{k=1}^{\infty}\frac{\left(  -1\right)  ^{k-1}}{k}\zeta\left(
p+k\right)  ,\text{ }p\geq1.
\]
The constant $\sigma_{p}$ occurs in several series and integral evaluations
(see \cite{Bo2018, CaC2, Cop, DBA}).

\begin{lemma}
\label{int-psi}For an integer $p\geq2,$ we have%
\[
\int_{1}^{\infty}\frac{\psi\left(  x+1\right)  }{x^{p}}dx=-\frac{\gamma}%
{p-1}+\left(  -1\right)  ^{p}\left(  \sigma_{p}-\zeta^{\prime}\left(
p\right)  \right)  +\sum_{j=2}^{p-1}\frac{\left(  -1\right)  ^{p-1-j}}%
{j-1}\zeta\left(  p+1-j\right)  .
\]

\end{lemma}

\begin{proof}
In view of the series representation \cite[Eq. 6.3.16]{AS}
\[
\psi\left(  x+1\right)  =-\gamma+\sum_{k=1}^{\infty}\frac{x}{k\left(
k+x\right)  }%
\]
we obtain
\[
\int_{1}^{n}\frac{\psi\left(  x+1\right)  }{x^{p}}dx=\frac{\gamma}{p-1}\left(
\frac{1}{n^{p+1}}-1\right)  +\sum_{k=1}^{\infty}\frac{1}{k}\int_{1}^{n}%
\frac{dx}{\left(  k+x\right)  x^{p-1}}.
\]
Considering the following partial fraction decomposition
\[
\frac{1}{x^{p-1}\left(  k+x\right)  }=\sum_{j=1}^{p-1}\frac{A_{j}}{x^{j}%
}+\frac{1}{\left(  -k\right)  ^{p-1}}\frac{1}{k+x},
\]
where $A_{j}=\left(  -1\right)  ^{p-1-j}/k^{p-j},$ $1\leq j\leq p-1$, we see
that%
\begin{align*}
\int_{1}^{n}\frac{\psi\left(  x+1\right)  }{x^{p}}dx  &  =\frac{\gamma}%
{p-1}\left(  \frac{1}{n^{p+1}}-1\right)  +\sum_{k=1}^{\infty}\sum_{j=2}%
^{p-1}\frac{\left(  -1\right)  ^{p-1-j}}{k^{p+1-j}}\frac{1}{j-1}\left(
1-\frac{1}{n^{j-1}}\right) \\
&  +\sum_{k=1}^{\infty}\frac{\left(  -1\right)  ^{p}}{k^{p}}\left(  \ln
\frac{n}{k+n}-\ln\frac{1}{k+1}\right)  .
\end{align*}
Letting $n\rightarrow\infty,$ with the use of \cite[Theorem 4]{DBA}%
\[
\sum_{k=1}^{\infty}\frac{\ln\left(  k+1\right)  }{k^{p}}=\sigma_{p}%
-\zeta^{\prime}\left(  p\right)  ,
\]
we arrive at the desired result.
\end{proof}

\begin{remark}
The formula given in Lemma \ref{int-psi} has also been recorded in the recent
paper \cite[Lemma 1]{Cop}, with a slightly different proof.
\end{remark}

We now recall the formula \cite[Theorem 2.1]{XL},
\begin{align}
\sum_{n=r}^{\infty}\frac{H_{n}}{\left(  n+1-r\right)  ^{p}}  &  =\frac{1}%
{2}\left(  p+2\right)  \zeta\left(  p+1\right)  -\frac{1}{2}\sum_{j=1}%
^{p-2}\zeta\left(  p-j\right)  \zeta\left(  j+1\right) \nonumber\\
&  -\sum_{m=1}^{p-1}\left(  -1\right)  ^{m}\zeta\left(  p+1-m\right)
H_{r-1}^{\left(  m\right)  }-\left(  -1\right)  ^{p}\sum_{m=1}^{r}\frac{H_{m}%
}{m^{p}} \label{xl}%
\end{align}
for $p\in\mathbb{N}\backslash\left\{  1\right\}  ,$ which appears in the
following evaluation formula for the constant $\gamma_{h^{\left(  r\right)  }%
}.$

\begin{theorem}
\label{teo1}Let $r$ be a non-negative integer. Then,%
\begin{align*}
\Gamma\left(  r\right)  \gamma_{h^{\left(  r\right)  }}  &  =\frac{1}{2}%
\gamma^{2}+\frac{1}{2}\zeta\left(  2\right)  -\frac{1}{2}+\int_{0}^{\infty
}\frac{\ln\left(  1+t^{2}\right)  }{\left(  e^{2\pi t}-1\right)  t}dt-\left(
\psi\left(  r\right)  +\gamma\right)  \gamma+r!-r\\
&  +\sum_{j=1}^{r-1}\frac{H_{j}}{j}+\sum_{j=1}^{r-1}%
\genfrac{[}{]}{0pt}{}{r}{j}%
\left(  E\left(  r,j\right)  +\frac{\psi\left(  r\right)  }{r-j}-\frac
{r}{r+1-j}-H_{r-1}\zeta\left(  r+1-j\right)  \right)  ,
\end{align*}
where%
\begin{equation}
E\left(  r,j\right)  =\sum_{k=r}^{\infty}\frac{H_{k}}{\left(  k+1-r\right)
^{r+1-j}}-\int_{1}^{\infty}\frac{\psi\left(  x+1\right)  }{x^{r+1-j}%
}dx.\nonumber
\end{equation}

\end{theorem}

\begin{proof}
From Lemma \ref{Top} and Lemma \ref{Int}, we have
\begin{align}
&  \Gamma\left(  r\right)  \left(  \sum_{k=1}^{n}\frac{h_{k}^{\left(
r\right)  }}{k^{r}}-\int_{1}^{n}\frac{h_{x}^{\left(  r\right)  }}{x^{r}%
}dx\right)  =\frac{\left(  H_{n}\right)  ^{2}+H_{n}^{\left(  2\right)  }}%
{2}-\left(  \psi\left(  r\right)  +\gamma\right)  H_{n}-\frac{1}{2}\left(  \ln
n\right)  ^{2}\nonumber\\
&  +\psi\left(  r\right)  \ln n-H_{r-1}\left(  H_{n+r-1}-H_{n}\right)
+\sum_{j=1}^{r-1}\frac{H_{j-1}}{j+n}-\frac{\left(  n+1\right)  ^{\overline
{r-1}}}{n^{r}}\nonumber\\
&  +\sum_{j=0}^{r-1}%
\genfrac{[}{]}{0pt}{}{r}{j}%
\left(  \sum_{k=1}^{n}\frac{H_{k+r-1}}{k^{r+1-j}}-\left(  \psi\left(
r\right)  +\gamma\right)  H_{n}^{\left(  r+1-j\right)  }\right)  +\sum
_{j=1}^{r-1}\frac{H_{j}}{j}-\left(  r+\frac{1}{2}\right)  \left(  1-\frac
{1}{n}\right) \nonumber\\
&  +r!-%
{\displaystyle\sum_{j=0}^{r-1}}
\genfrac{[}{]}{0pt}{0}{r}{j}%
\left\{  \frac{r}{r+1-j}\left(  1-\frac{1}{n^{r+1-j}}\right)  -\frac
{\psi\left(  r\right)  }{r-j}\left(  1-\frac{1}{n^{r-j}}\right)  \right\}
\nonumber\\
&  +\int_{0}^{\infty}\frac{\ln\left(  1+t^{2}\right)  -\ln\left(  1+\left(
t/n\right)  ^{2}\right)  }{\left(  e^{2\pi t}-1\right)  t}dt-%
{\displaystyle\sum_{j=0}^{r-1}}
\genfrac{[}{]}{0pt}{0}{r}{j}%
\int_{1}^{n}\frac{\psi\left(  x+1\right)  }{x^{r+1-j}}dx. \label{7}%
\end{align}
Considering the well-known properties%
\begin{equation}
H_{n}=\ln n+\gamma+\frac{1}{2n}-\frac{1}{12n^{2}}+O\left(  n^{-4}\right)
\text{ and }\lim_{n\rightarrow\infty}\left(  H_{n+r-1}-H_{n}\right)  =0,
\label{Hn}%
\end{equation}
and then letting $n\rightarrow\infty$ in (\ref{7}) give the desired formula
after some manipulations.
\end{proof}

With the help of Lemma \ref{int-psi} and (\ref{xl}), the statement of Theorem
\ref{teo1} can be equivalently written as
\begin{align*}
&  \Gamma\left(  r\right)  \gamma_{h^{\left(  r\right)  }}=\frac{1}{2}%
\gamma^{2}+\frac{1}{2}\zeta\left(  2\right)  -\frac{1}{2}+\int_{0}^{\infty
}\frac{\ln\left(  1+t^{2}\right)  }{\left(  e^{2\pi t}-1\right)  t}%
dt+\sum_{j=1}^{r-1}\frac{H_{j}}{j}-\left(  \psi\left(  r\right)
+\gamma\right)  \gamma+r!\\
&  -r+\sum_{j=1}^{r-1}%
\genfrac{[}{]}{0pt}{}{r}{j}%
\left\{
\begin{array}
[c]{l}%
\frac{r+3-j}{2}\zeta\left(  r+2-j\right)  -\frac{1}{2}\sum\limits_{v=1}%
^{r-j-1}\zeta\left(  r+1-j-v\right)  \zeta\left(  v+1\right) \\
-\sum\limits_{v=2}^{r-j}\left(  -1\right)  ^{v}\zeta\left(  r+2-j-v\right)
\left(  H_{r-1}^{\left(  v\right)  }+\dfrac{\left(  -1\right)  ^{r-j}}%
{v-1}\right)  +\frac{H_{r-1}}{r-j}\\
+\left(  -1\right)  ^{r-j}\left(  \sigma_{r+1-j}-\zeta^{\prime}\left(
r+1-j\right)  +\sum\limits_{v=1}^{r-1}\frac{H_{v}}{v^{r+1-j}}\right)
-\frac{r}{r+1-j}%
\end{array}
\right\}  .
\end{align*}
By straightforward computations we observe that
\begin{align*}
\gamma_{h^{\left(  0\right)  }}  &  =-\gamma\lim_{r\rightarrow0}\frac
{\psi\left(  r\right)  }{\Gamma\left(  r\right)  }=\gamma,\\
\gamma_{h^{\left(  1\right)  }}  &  =\frac{1}{2}\gamma^{2}+\frac{1}{2}%
\zeta\left(  2\right)  -\frac{1}{2}+\int_{0}^{\infty}\frac{\ln\left(
1+t^{2}\right)  }{\left(  e^{2\pi t}-1\right)  t}dt,\\
\gamma_{h^{\left(  2\right)  }}  &  =\gamma_{h^{\left(  1\right)  }}%
-\gamma-\sigma_{2}+\zeta^{^{\prime}}\left(  2\right)  +2\zeta\left(  3\right)
,\\
\gamma_{h^{\left(  3\right)  }}  &  =\frac{1}{2}\gamma_{h^{\left(  2\right)
}}-\frac{1}{4}\gamma+2\zeta\left(  3\right)  +\frac{1}{72}\pi^{2}\left(
\pi^{2}-27\right)  +\frac{5}{4}-\zeta^{\prime}\left(  3\right)  +\zeta
^{\prime}\left(  2\right)  -\sigma_{2}+\sigma_{3}.
\end{align*}

Moreover, according to table on page \pageref{liste} it seems that
$\gamma_{h^{\left(  r\right)  }}\leq\gamma_{h^{\left(  r+1\right)  }}$ for
$r\geq1$. Our attempts to prove this observation were not successful,
therefore it remains an open question.

\begin{remark}
Combining Connon's \cite{Co} results (3.9) and (3.46) gives (with a misprint
corrected)%
\[
\int_{0}^{\infty}\frac{\ln\left(  1+t^{2}\right)  }{\left(  e^{2\pi
t}-1\right)  t}dt=\sigma_{1}+\frac{1}{2}-\zeta\left(  2\right)  +\gamma_{1}.
\]

\end{remark}

\subsection{The constants $\overline{\gamma}_{h^{\left(  r\right)  }}$}

The following lemmas are the\ analogues of Lemma \ref{Top} and Lemma \ref{Int}.

\begin{lemma}
\label{Top-}Let $r$ be a non-negative integer. Then,%
\begin{align*}
\Gamma\left(  r\right)  \sum_{k=1}^{n}\frac{h_{k}^{\left(  r\right)  }%
}{k^{\overline{r}}}  &  =\frac{\left(  H_{n}\right)  ^{2}+H_{n}^{\left(
2\right)  }}{2}-H_{r-1}\left(  H_{n+r-1}-H_{n}\right)  -\left(  \psi\left(
r\right)  +\gamma\right)  H_{n}\\
&  +\sum_{j=1}^{r-1}\frac{H_{j}}{j}+\sum_{j=1}^{r-1}\frac{H_{j-1}}{j+n}.
\end{align*}

\end{lemma}

\begin{proof}
The proof follows from (\ref{2}), $\psi\left(  r\right)  +\gamma=H_{r-1}$ and
(\ref{top1}).
\end{proof}

The proof of the following lemma is similar to the proof of Lemma \ref{Int}.
Thus we omit it.

\begin{lemma}
\label{Int-}Let $r$ be a non-negative integer. Then,%
\begin{align*}
\Gamma\left(  r\right)  \int_{1}^{n}\frac{h_{x}^{\left(  r\right)  }%
}{x^{\overline{r}}}dx  &  =\frac{1}{2}\left(  \ln n\right)  ^{2}-\frac{1}%
{2n}+\frac{1}{2}-\int_{0}^{\infty}\frac{\ln\left(  1+t^{2}\right)  -\ln\left(
1+\left(  t/n\right)  ^{2}\right)  }{\left(  e^{2\pi t}-1\right)  t}dt\\
&  +\sum\limits_{j=1}^{r-1}\frac{\ln n+\ln\left(  1+j\right)  -\ln\left(
n+j\right)  }{j}-\psi\left(  r\right)  \ln n.
\end{align*}

\end{lemma}

To present the counterpart of Theorem \ref{teo1} we use Lemma \ref{Top-} and
Lemma \ref{Int-}. Thus,
\begin{align*}
&  \Gamma\left(  r\right)  \left(  \sum_{k=1}^{n}\frac{h_{k}^{\left(
r\right)  }}{k^{\overline{r}}}-\int_{1}^{n}\frac{h_{x}^{\left(  r\right)  }%
}{x^{\overline{r}}}dx\right) \\
&  \ =\frac{\left(  H_{n}\right)  ^{2}+H_{n}^{\left(  2\right)  }}{2}-\frac
{1}{2}\left(  \ln n\right)  ^{2}+\psi\left(  r\right)  \ln n-\left(
\psi\left(  r\right)  +\gamma\right)  H_{n}\\
&  \quad-\frac{1}{2}+\int_{0}^{\infty}\frac{\ln\left(  1+t^{2}\right)
-\ln\left(  1+\left(  t/n\right)  ^{2}\right)  }{\left(  e^{2\pi t}-1\right)
t}dt+\sum_{j=1}^{r-1}\frac{H_{j}-\ln\left(  1+j\right)  }{j}\\
&  \quad+\sum_{j=1}^{r-1}\frac{H_{j-1}}{j+n}-H_{r-1}\left(  H_{n+r-1}%
-H_{n}\right)  -\sum\limits_{j=1}^{r-1}\frac{\ln n-\ln\left(  n+j\right)  }%
{j}+\frac{1}{2n}.
\end{align*}
We now let $n\rightarrow\infty$, Using (\ref{Hn}) we deduce that%
\[
\Gamma\left(  r\right)  \overline{\gamma}_{h^{\left(  r\right)  }}%
=\frac{\gamma^{2}}{2}+\frac{\zeta\left(  2\right)  }{2}-\frac{1}{2}+\int%
_{0}^{\infty}\frac{\ln\left(  1+t^{2}\right)  }{\left(  e^{2\pi t}-1\right)
t}dt-\left(  \psi\left(  r\right)  +\gamma\right)  \gamma+\sum_{j=1}%
^{r-1}\frac{H_{j}-\ln\left(  1+j\right)  }{j}.
\]
Thus, we have obtained the following theorem.

\begin{theorem}
\label{teo1-}Let $r$ be a non-negative integer. Then,%
\[
\overline{\gamma}_{h^{\left(  r\right)  }}=\frac{1}{\Gamma\left(  r\right)
}\left(  \gamma_{h^{\left(  1\right)  }}-\left(  \psi\left(  r\right)
+\gamma\right)  \gamma+\sum_{j=1}^{r-1}\frac{H_{j}-\ln\left(  1+j\right)  }%
{j}\right)  .
\]

\end{theorem}


\begin{thebibliography}{99}                                                                                               %

\bibitem{AS}
M.~Abramowitz and I.~A. Stegun.
\newblock {\em Handbook of Mathematical Functions}.
\newblock National Bureau of Standards, New York, 1965.

\bibitem{Ad}
V.~S. Adamchik.
\newblock {On Stirling numbers and Euler sums}.
\newblock {\em J. Comput. Appl. Math.}, 79:119--130, 1997.

\bibitem{Be}
B.~C. Berndt.
\newblock {On the Hurwitz zeta-function}.
\newblock {\em Rocky Mountain Journal of Mathematics}, 2, No. 1:151--157, 1972.

\bibitem{Bl}
I.~V. Blagouchine.
\newblock {A theorem for the closed-form evaluation of the first generalized
  Stieltjes constant at rational arguments and some related summations}.
\newblock {\em J. Number Theory}, 148:537--592, 2015.

\bibitem{Bl-1}
I.~V. Blagouchine.
\newblock {Expansions of generalized Euler's constants into the series of
  polynomials in $\pi^{-2}$ and into the formal enveloping series with rational
  coefficients only}.
\newblock {\em J. Number Theory}, 158:365--396, 2016.

\bibitem{BlC}
I.~V. Blagouchine and M.~A. Coppo.
\newblock {A note on some constants related to the zeta-function and their
  relationship with the Gregory coefficients}.
\newblock {\em The Ramanujan Journal}, 47(2):457--473, 2018.

\bibitem{Bo2018}
K.~N. Boyadzhiev.
\newblock A special constant and series with zeta values and harmonic numbers.
\newblock {\em Gazeta Matematica Seria A}, 115:1--16, 2018.

\bibitem{Bo}
K.~N. Boyadzhiev.
\newblock {New series identities with Cauchy, Stirling, and harmonicnumbers,
  and Laguerre polynomials}.
\newblock {\em J. Integer Sequences}, 23:Article20.11.7, 2020.

\bibitem{BK}
K.~N. Boyadzhiev and L.~Karg\i n.
\newblock {New series with Cauchy and Stirling numbers, Part 2.}
\newblock {\em arXiv preprint}, arXiv:2103:11960, 2021.

\bibitem{Briggs}
W.E. Briggs and S.~Chowla.
\newblock The power series coefficients of $\zeta\left( s\right) $.
\newblock {\em Am. Math. Mon.}, 62:323--325, 1955.

\bibitem{CaC}
B.~Candelpergher and M.~A. Coppo.
\newblock {A new class of identities involving Cauchy numbers, harmonic numbers
  and zeta values}.
\newblock {\em The Ramanujan Journal}, 27(3):305--328, 2012.

\bibitem{CaC2}
B.~Candelpergher and M-A. Coppo.
\newblock Laurent expansion of harmonic zeta functions.
\newblock {\em J. Math. Anal. App.}, 491:124309, 2020.

\bibitem{CK}
T.~Chatterjee and S.~S. Khurana.
\newblock {Shifted Euler constants and a generalization of Euler-Stieltjes
  constants}.
\newblock {\em Journal of Number Theory}, 204:185--210, 2019.

\bibitem{Co}
D.~F. Connon.
\newblock On an integral involving the digamma function.
\newblock {\em arXiv.org}, 1212.1432:https://arxiv.org/abs/1212.1432, 2012.

\bibitem{CG}
John~H. Conway and Richard Guy.
\newblock {\em The book of numbers}.
\newblock Springer Science \& Business Media, 1998.

\bibitem{Co2019}
M.~A. Coppo.
\newblock A note on some alternating series involving zeta and multiple zeta
  values.
\newblock {\em Journal of Mathematical Analysis and Applications},
  475.2:1831--1841, 2019.

\bibitem{Cop}
M.~A. Coppo and B.~Candelpergher.
\newblock {A note on some formulae related to Euler sums}.
\newblock pages https://hal.univ--cotedazur.fr/hal--03170892.

\bibitem{CY}
M.~A. Coppo and P.~T. Young.
\newblock {On shifted Mascheroni series and hyperharmonic numbers}.
\newblock {\em Journal of Number Theory}, 169:1--20, 2016.

\bibitem{DeT}
D.~W. DeTemple.
\newblock {A quicker convergence to Euler's constant}.
\newblock {\em Am. Math. Mon.}, 100:468--470, 1993.

\bibitem{DB}
A.~Dil and K.N. Boyadzhiev.
\newblock {Euler sums of hyperharmonic numbers}.
\newblock {\em Journal of Number Theory}, 147:490--498, 2015.

\bibitem{DBA}
A.~Dil, K.N. Boyadzhiev, and I.A. Aliyev.
\newblock {On values of the Riemann zeta function at positive integers}.
\newblock {\em Lith. Math. J.}, 60:9--24, 2020.

\bibitem{Di}
K.~Dilcher.
\newblock {Generalized Euler constants for arithmetical progressions}.
\newblock {\em Mathematics of Computation}, 59:259--282, 1992.

\bibitem{FPS}Farr, R. E., Pauli, S., Saidak, F. (2018). On fractional Stieltjes constants. Indagationes Mathematicae, 29(5), 1425-1431.

\bibitem{GKP}
R.~L. Graham, D.~E. Knuth, and O.~Patashnik.
\newblock {\em {Concrete Mathematics}}.
\newblock AddisonWesley, Reading, MA, 1994.

\bibitem{JB}
F.~Johansson and I.~Blagouchine.
\newblock {Computing Stieltjes constants using complex integration}.
\newblock {\em Mathematics of Computation}, 88:1829--1850, 2019.

\bibitem{J}
C.~Jordan.
\newblock {\em Calculus of Finite Differences}.
\newblock Chelsea, New York, 1965.

\bibitem{Ka}
E.~A. Karatsuba.
\newblock {On the computation of the Euler constant $\gamma$ }.
\newblock {\em Numer. Algorithms}, 24:83--97, 2000.

\bibitem{Kn}
K.~Knopp.
\newblock {\em Theory and Application of Infinite Series}.
\newblock Blackie, London, 1951.

\bibitem{La}
V.~Lampret.
\newblock {A double inequality for a generalized-Euler-constant function}.
\newblock {\em J.Math. Anal. Appl.}, 381:155--165, 2011.

\bibitem{Lehmer}
D.~H. Lehmer.
\newblock Euler constants for arithmetical progressions.
\newblock {\em Acta Arithmetica}, 27:125--142, 1975.

\bibitem{Lu}
D.~Lu.
\newblock {A new quicker sequence convergent to Euler's constant}.
\newblock {\em J. Number Theory}, 136:----------, 2014.

\bibitem{LuSY}
D.~Lu, L.~Song, and Y.~Yu.
\newblock {Some new continued fraction approximation of Euler's constant}.
\newblock {\em J. Number Theory}, 147:69--80, 2015.

\bibitem{Ma}
L.~Mascheronio.
\newblock {\em {Adnotationes ad calculum integralem Euleri in quibus nonnulla
  problemata ab Eulero proposita resolvuntur}}.
\newblock Ex Typographia Petri Galeatii, Ticini, 1790.

\bibitem{M}
I.~Mez\H{o}.
\newblock Analytic extension of hyperharmonic numbers.
\newblock {\em Online Journal of Analytic Combinatorics}, 4, 2009.

\bibitem{MD}
I.~Mez\H{o} and A.~Dil.
\newblock {Hyperharmonic series involving Hurwitz zeta function}.
\newblock {\em J. Number Theory}, 130:360--369, 2010.

\bibitem{Mo}
C.~Mortici.
\newblock {Improved convergence towards generalized Euler-Mascheroni constant}.
\newblock {\em Appl. Math. Comput.}, 215:3443--3448, 2010.

\bibitem{MS}
M.~R. Murty and N.~Saradha.
\newblock {Euler-Lehmer constants and a conjecture of Erd\"{o}s}.
\newblock {\em J. Number Theory}, 130:2671--2682, 2010.

\bibitem{KCDC}
L.~Karg\i n, M.~Can, A.~Dil, and M.~Cenkci.
\newblock {On evaluations of Euler-type sums of hyperharmonic number}.
\newblock {\em Bull. Malays. Math. Sci. Soc.,
  https://doi.org/10.1007/s40840-021-01179-8}.

\bibitem{S2007}
A.~S\^{\i}nt\u{a}m\u{a}rian.
\newblock { A generalization of Euler's constant}.
\newblock {\em Numerical Algorithms}, 46(2):141--151, 2007.

\bibitem{S2013}
A.~S\^{\i}nt\u{a}m\u{a}rian.
\newblock Euler's constant, sequences and some estimates.
\newblock {\em Surveys in Mathematics \& its Applications}, 8:-----, 2013.

\bibitem{SoH}
J.~Sondow and P.~Hadjicostas.
\newblock {The generalized-Euler-constant function $\gamma\left( z\right) $ and
  a generalization of Somos's quadratic recurrence constant}.
\newblock {\em J. Math. Anal. Appl.}, 332:292--314, 2007.

\bibitem{SC}
H.~M. Srivastava and J.~Choi.
\newblock {\em Series Associated with the Zeta and Related Functions}.
\newblock Kluwer Academic Publishers, Dordrecht-Boston-London, 2001.

\bibitem{We}
E.~Weisstein.
\newblock Maclaurin-cauchy theorem.
\newblock {\em http://mathworld.wolfram.com/ Maclaurin-CauchyTheorem.html},
  accessed May 23, 2021.

\bibitem{Wi}
J.~R. Wilton.
\newblock A note on th\'{e} coefficients in the expansion of $\zeta(s,x)$ in
  powers of $s-1$.
\newblock {\em Quart. J. Pure Appi. Math.}, 50:329--332, 1927.

\bibitem{Xu}
A.~Xu.
\newblock {Approximations of the generalized-Euler-constant function and the
  generalized Somos' quadratic recurrence constant}.
\newblock {\em Journal of Inequalities and Applications}, 2019:198, 2019.

\bibitem{XL}
C.~Xu and Z.~Li.
\newblock {Tornheim type series and nonlinear Euler sums}.
\newblock {\em J. Number Theory}, 174:40--67, 2017.

\bibitem{Yo}
R.~M. Young.
\newblock Euler'sconstant.
\newblock {\em Math. Gaz.}, 75:187--190, 1991.

\end{thebibliography}
\end{document}